\documentclass[11pt]{amsart}
\usepackage{amssymb}
\usepackage{amsmath}
\usepackage[all, cmtip, line]{xy}
\newtheorem{thm}{Theorem}[section]
\newtheorem{prop}[thm]{Proposition}
\newtheorem{cor}[thm]{Corollary}
\newtheorem{lem}[thm]{Lemma}
\newtheorem{rem}[thm]{Remark}

\newcommand{\C}{{\mathcal C}}
\newcommand{\D}{{\mathcal D}}

\newcommand{\E}{{\mathcal E}}
\newcommand\Rep{\operatorname{Rep}}
\newcommand\Irr{\operatorname{Irr}}
\newcommand\FPdim{\operatorname{FPdim}}

\newcommand\vect{\operatorname{Vec}}
\newcommand\SuperV{\operatorname{SuperVec}}

\newcommand\Hom{\operatorname{Hom}}

\newcommand\Pic{\operatorname{G}}

\begin{document}

\title[Tannakian subcategories]{Existence of Tannakian subcategories and its applications}
\author[dong]{Jingcheng Dong}
\address{College of Engineering, Nanjing Agricultural University, Nanjing 210031, P.R.
China}
\email[Dong]{dongjc@njau.edu.cn}
\author[Dai]{Li Dai}
\email[Dai]{daili1980@njau.edu.cn}

\keywords{solvable fusion category, group-theoretical fusion category; equivariantization; Frobenius-Perron dimension}

\subjclass[2010]{18D10; 16T05}

\date{\today}


\begin{abstract}
We study several classes of braided fusion categories, and prove that they all contain nontrivial Tannakian subcategories. As applications, we classify some fusion categories in terms of  solvability and group-theoreticality.
\end{abstract}
 \maketitle



\section{Introduction}\label{sec1}
Recall that a symmetric fusion category is called Tannakian if it is equivalent to the representation category of a finite group as a symmetric fusion category. Tannakian categories play an important role in the study and classification of fusion categories.
See
\cite{bruguieres2014central,bruillard2013classification,burciu2013fusion,etingof2011weakly,naidu2009fusion,nikshych2008non}
 for examples.

 The aim of this paper is to establish the existence of Tannakian subcategories of some braided fusion categories so as to study the classification of these fusion categories. Our first result is:

 \begin{thm}\label{thm001}
 Let $\C$ be an integral braided fusion category.

 (1)\quad (Theorem \ref{thm32}) If the Frobenius-Perron dimensions of all noninvertible simple objects have a common prime divisor $p$ then the largest pointed fusion category $\C_{pt}$ contains a nontrivial Tannakian subcategory of $\C$.

 (2)\quad (Theorem \ref{thm42}) Let $1<p<q$ be positive integers such that ${\rm gcd}(p,q)=1$, and let $\C$ be of type $(1,m;p,n;q,s)$. Then $\C$ contains a nontrivial Tannakian subcategory in any of the following cases:

(i)\quad $p^2\leq q$;

(ii)\quad $p,q$ are powers of two distinct prime numbers.
 \end{thm}

The definition of type of a fusion category will be given in Subsection \ref{subsec21}.


Obviously, part (1) of Theorem \ref{thm001} can be applied to integral braided fusion category $\C$ of type $(1,m;s,n)$. Thus such fusion category must be a $G$-equivariantization of a fusion category $C_G$, where $\Rep(G)\subseteq \C$ a nontrivial Tannakian subcategory, and $C_G$ is the corresponding de-equivariantization.

\medbreak
In \cite{gelaki2009centers}, Gelaki, Naidu and Nikshych constructed a class of braided fusion categories of type $(1,m;p,n;q,s)$, where $p<q$ are prime numbers. They are solvable but not group-theoretical. We prove that braided fusion categories of this type are always solvable. Our second result is:

\begin{thm}\label{thm01}
(1)\quad(Corollary \ref{cor34}) Let $p$ be a prime number and let $\C$ be an integral braided fusion category such that the Frobenius-Perron dimensions of simple objects are powers of $p$. Then $\C$ is solvable.

(2)\quad(Corollary \ref{cor43}) Let $\C$ be a braided fusion category. Suppose that $\C$ is of type $(1,m;p,n;q,s)$ then $\C$ is solvable, where $p<q$ are prime numbers.
\end{thm}

\begin{rem}
Parts (1) and (2) of Theorem \ref{thm01} were respectively obtained in \cite{natale2013weakly} and \cite{natale2014graphs} by different methods.
\end{rem}

\medbreak
For a prime number $p$, the fusion categories of type $(1,m;p,n)$ have been  studied in several settings in \cite{natale2012fusion,natale2012solvability}. Our results below generalize some of them. We pay more attention on when a fusion category is group-theoretical. Our third result is:

\begin{thm} (Theorem \ref{thm36}) Let $\C$ be an integral braided fusion category such that  the Frobenius-Perron dimensions of its simple objects are powers of a prime number $p$. If $\FPdim(\C_{pt})=p$ then $\C$ is group-theoretical. Moreover, $\C$ is of type $(1,p;p,m)$ or $(1,p)$.
\end{thm}

In the last part of the paper, we apply above results to odd-dimensional braided fusion categories. We prove the following result.

\begin{thm} (Theorem \ref{thm51}) Let $\C$ be an integral braided fusion category. Suppose that $\FPdim(\C)$ is odd and $\FPdim(\C)<765$. Then $\C$ is solvable.
\end{thm}

\medbreak
The paper is organized as follows. In Section \ref{sec2}, we give some basic definitions and results which will be used throughout. For other basics on fusion categories, we refer the reader to \cite{drinfeld2010braided,etingof2005fusion}. In Section \ref{sec3}, we study braided fusion categories such that the Frobenius-Perron dimensions of noninvertible simple objects have a common prime divisor. In Section \ref{sec4}, we study braided fusion categories such that the Frobenius-Perron dimensions of noninvertible simple objects are two coprime integers. In the last section, we study odd-dimensional fusion categories with Frobenius-Perron dimension less than $765$.

\section{Preliminaries}\label{sec2}
\subsection{Fusion categories}\label{subsec21}
Let $k$ be an algebraically closed field of characteristic $0$. Recall that a fusion category over $k$ is a $k$-linear semisimple rigid tensor category $\C$ with finitely many isomorphism classes of simple objects, finite dimensional Hom spaces, and the unit object $\textbf{1}$ is simple. Throughout this paper, $\Irr(\C)$ denotes the set of isomorphism classes of
simple objects of $\C$, and $\Irr_{\alpha}(\C)$ denotes the set of isomorphism classes of simple objects of Frobenius-Perron dimension $\alpha$, where $\alpha$ is a positive real number.

\medbreak
A fusion subcategory of $\C$ is a full tensor subcategory $\D$ such that if $X$ is an object of $\C$ isomorphic to a direct summand of an object of $\D$, then $X$ lies in $\D$. If $\D$ is a fusion subcategory of $\C$ then the quotient $\FPdim(\C)/\FPdim(\D)$ is an algebraic integer \cite[Proposition 8.15]{etingof2005fusion}, where $\FPdim(\C)$ denotes the Frobenius-Perron dimension of $\C$. For an object $X$ of $\C$, we will use $\FPdim(X)$ to denote its Frobenius-Perron dimension.

A fusion category $\C$ is called weakly integral if the Frobenius-Perron dimension $\FPdim (\C)$ is a natural number. It is called integral if $\FPdim (X)$ is a natural number for every object $X$ of $\C$. Obviously,  if $\C$ is weakly integral then $\FPdim (X)^2$ is a natural number. By \cite[Theorem 8.33]{etingof2005fusion}, every integral fusion category is equivalent to the representation category of a finite dimensional quasi-Hopf algebra.

\medbreak
A simple object $X$ is invertible if and only if its Frobenius-Perron dimension is $1$. Let $\Pic(\C)$ denote the set of isomorphism classes of invertible simple objects of $\C$. Then $\Pic(\C)$ is a group and generates a fusion subcategory $\C_{pt}$ of $\C$. This is the unique largest pointed fusion subcategory of $\C$. Here, a fusion category is called pointed if every simple object has Frobenius-Perron dimension $1$. So the order of $\Pic(\C)$ coincides with the Frobenius-Perron dimension of $\C_{pt}$ and hence it divides $\FPdim(\C)$.

\medbreak
Let $1=d_0, d_1,\cdots, d_s$ be positive real numbers with $1 = d_0 < d_1< \cdots < d_s$, and $n_0,n_1,\cdots,n_s$ be positive integers. A fusion category $\C$ is called of type $(d_0,n_0;d_1,n_1;\cdots;d_s,n_s)$ if $n_i$ is the number of the nonisomorphic simple objects of Frobenius-Perron dimension $d_i$, for all $i = 0, \cdots, s$.

Let $X$ and $Y$ be objects of $\C$. We denote the multiplicity of $X$ in $Y$ by $m(X,Y)$. Hence $$m(X, Y)=\dim \Hom_\C(X, Y)\mbox{\,\,and\,\,} Y = \sum_{X \in \Irr
(\C)} m(X, Y) X.$$

Let $X, Y, Z$ be objects of $\C$. Then we have $m(X,Y)=m(X^*,Y^*)$ and
$$m(X,Y\otimes Z)=m(Y^*,Z\otimes X^*)=m(Y,X\otimes Z^*).$$

Let $X, Y \in \Irr(\C)$. Then for each $g \in \Pic(\C)$ we have
\begin{center}
$m(g,X\otimes Y)=\left\{ {\begin{array}{l}
 1,\mbox{\,if\,\,} Y\cong X^*\otimes g; \\
 0,\mbox{otherwise}. \\
 \end{array}} \right.$
\end{center}

In particular, $m(g,X\otimes Y)=0$ if $\FPdim X\neq \FPdim Y$. Let $X\in
\Irr(\C)$. Then, for every $g \in \Pic(\C)$, $m(g,X\otimes X^{*})>0$ if and
only if $m(g,X\otimes X^{*})= 1$ if and only if $g\otimes X\cong X$. The set of
isomorphism classes of such invertible objects will be denoted by
$G[X]$. Thus $G[X]$ is a subgroup of $\Pic(\C)$ whose order divides
$(\FPdim X)^2$ \cite[Lemma 2.2]{dong2012frobenius}. In particular, for all $X\in \Irr(\C)$, we have a
relation
\begin{equation}\label{eq1}
\begin{split}
X\otimes X^*=\bigoplus_{g\in G[X]}g\oplus\bigoplus_{Y\in \Irr(\C),\FPdim Y >1} m(Y, X\otimes X^*) Y.
\end{split}
\end{equation}

\medbreak
Actually, there is an action of the group $\Pic(\C)$ on the set $\Irr(\C)$ by left
tensor multiplication. This action preserves Frobenius-Perron dimensions. So, for every $X\in \Irr(\C)$, $G[X]$ is the stabilizer of $X$ in $\Pic(\C)$.

\begin{lem}\label{lem21}
Let $\C$ be a weakly integral fusion category such that the Frobenius-Perron dimensions of all noninvertible simple objects have a common prime divisor $p$. Then $p$ divides the order of $\Pic(\C)$.
\end{lem}
\begin{proof}
Let $X$ be a noninvertible simple object of $\C$. Then Equation \ref{eq1} shows that the order of $G[X]$ is divisible by $p$. It follows that $p$ divides the order of $\Pic(\C)$ since $G[X]$ is a subgroup of $\Pic(\C)$.
\end{proof}

It is interesting that a counterexample of the fusion category given in Lemma \ref{lem21} also contains nontrivial invertible simple objects. Also, fusion categories of the type given in the following proposition really exists. See \cite[Section 5C]{gelaki2009centers} for details.

\begin{prop}\label{pro22}
Let $1<p<q$ be positive integers such that ${\rm gcd}(p,q)=1$, and let $\C$ be a fusion category of type $(1,m;p,n;q,s)$. Then $m>1$.
\end{prop}
\begin{proof}
Suppose on the contrary that $m=1$. Let $X$ be an element of $\Irr_p(\C)$. Then
\begin{equation}\label{eq2}
\begin{split}
X\otimes X^*=\textbf{1}\oplus\bigoplus_{X_i\in \Irr_p(\C)}m(X_i, X\otimes X^*) X_i\oplus\bigoplus_{Y_j\in \Irr_q(\C)} m(Y_j, X\otimes X^*) Y_j.
\end{split}\nonumber
\end{equation}
There must exist some $Y_j\in \Irr_q(\C)$ such that $m(Y_j, X\otimes X^*)>0$, otherwise the decomposition above can not hold true. We choose one such $Y_j$ and denote it by $Y$. Then
 \begin{equation}\label{eq3}
\begin{split}
0<m(Y, X\otimes X^*)=m(X,Y\otimes X)<\frac{p^2}{q}<q.
\end{split}
\end{equation}

We claim that $Y\otimes X$ can not contain elements of $\Irr_q(\C)$. Indeed, if $Y\otimes X$ contains $Z\in \Irr_q(\C)$ such that $m(Z,Y\otimes X)>0$ then we have an equation
\begin{equation}\label{eq4}
\begin{split}
pq=mp+nq,
\end{split}
\end{equation}
where $m,n$ are positive integers. This induces a contradiction $p(q-m)=nq$. Notice that $\textbf{1}$ can not appear in the decomposition of $Y\otimes X$ since $X$ and $Y$ are of different Frobenius-Perron dimensions. So Equation \ref{eq4} holds true.

Equation \ref{eq3} implies that, besides $X$, $Y\otimes X$ also contains other elements of
$\Irr_p(\C)$. Thus we may write
\begin{equation}\label{eq5}
\begin{split}
Y\otimes X=m(X,Y\otimes X)X\oplus\bigoplus_{X_i\neq X,X_i\in\Irr_p(\C)}m(X_i,Y\otimes X)X_i.
\end{split}
\end{equation}

We choose one $X\neq X_i\in\Irr_p(\C)$ such that $m(X_i,Y\otimes X)>0$ and denote it by $X'$. Then $m(X',Y\otimes X)=m(Y,X'\otimes X^*)>0$. Since $X'\neq X$, $\textbf{1}$ can not appear in the decomposition of $X'\otimes X^*$. It follows from the decomposition of $X'\otimes X^*$ that we have an equation $p^2=mq+np$, where $m,n$ are positive integers. Since $p<q$, this induces a contradiction that $p(p-n)=mq$.
\end{proof}

Let $H$ be a semisimple Hopf algebra over $k$. Then \cite[Lemma 11]{zhu1993finite} states that if $G(H^*)$ is trivial then $H$ has at least $4$ nonisomorphic simple modules of different dimensions. A similar result also appears in \cite[Theorem 2.2.3]{natale1999semisimple}. Our result here is an analogous in the fusion category context.

\subsection{Extensions}
Let $G$ be a finite group. A $G$-grading of $\C$ is a decomposition of $\C$ as a direct sum of full Abelian subcategories $\C=\otimes_{g\in G}\C_g$, such that $\C_g^*=\C_{g^{-1}}$ and the tensor product $\otimes:\C\otimes\C\to\C$ maps $\C_g\times \C_h$ to $\C_{gh}$. In general, $\C_g$ is not a fusion subcategory of $\C$. But the neutral component $\C_e$ is a fusion subcategory of $\C$, where $e$ is the identity element of $G$. The grading $\C=\otimes_{g\in G}\C_g$ is called faithful if $\C_g\neq0$ for all $g\in G$. We say that $\C$ is an $G$-extension of $\D$ if the grading is faithful and $\C_e=\D$. If $\C$ is a $G$-extension of a fusion category $\D$ then the Frobenius-Perron dimensions of $\C_g$ are equal for all $g\in G$, and $\FPdim (\C)=|G|\FPdim(\D)$ \cite[Proposition 8.20]{etingof2005fusion}.

\medbreak
Every fusion category $\C$ has a canonical faithful grading $\C=\otimes_{g\in \mathcal{U}(\C)}\C_g$, whose neutral component $\C_e$ equals to the adjoint subcategory $\C_{ad}$, where $\C_{ad}$ is the smallest fusion subcategory of $\C$ containing all objects $X\otimes X^*$ for all $X\in \Irr(\C)$. The group $\mathcal{U}(\C)$ is called the universal grading group of $\C$ \cite{gelaki2008nilpotent}. Any faithful grading $\C=\otimes_{g\in G}\C_g$ comes from a group epimorphism $\mathcal{U}(\C)\to G$. So $\C_e\supseteq \C_{ad}$ for any extension of $\C$ by a group $G$.

\medbreak
Suppose that $\C$ is a weakly integral fusion category. Then either $\C$ is integral or $\C$ is a $\mathbb{Z}_2$-extension of an integral fusion subcategory \cite[Theorem 3.10]{gelaki2008nilpotent}.

\subsection{Morita equivalence}
Let $\C$ be a fusion category
and let $\mathcal{M}$ be an indecomposable right $\C$-module category. Let $C^*_{\mathcal{M}}$ denote the category of $\C$-module endofunctors of $\mathcal{M}$. Then  $C^*_{\mathcal{M}}$ is a fusion category, called the dual of $\C$ with respect
to $\mathcal{M}$ \cite{etingof2005fusion,ostrik2003module}. Two fusion categories $\C$ and $\D$ are Morita equivalent if $D$ is equivalent to $C^*_{\mathcal{M}}$ for some indecomposable right $\C$-module category $\mathcal{M}$.

A fusion category $\mathcal{\C}$ is said to be (cyclically) nilpotent if there is a sequence of
fusion categories $\mathcal{\C}_0 ={\rm Vec}, \mathcal{\C}_1, \cdots, \mathcal{\C}_n =\mathcal{\C}$ and a sequence of finite (cyclic) groups $G_1,\cdots, G_n$ such that $\mathcal{\C}_i$ is obtained from $\mathcal{\C}_{i-1}$ by a $G_i$-extension. Let $\C^{(0)}=\C, \C^{(1)}=\C_{ad}$ and $\C^{(i)}=(\C^{(i-1)})_{ad}$ for
every integer $i\geq1$. It is clear that $\C$ is nilpotent if and only if the sequence $\C^{(0)}\supseteq \C^{(1)}\supseteq\cdots\supseteq \C^{(i)}\supseteq\cdots$ converges to $\vect$.

A fusion category is called group-theoretical if it is Morita equivalent to a pointed fusion category. A fusion category is called weakly group-theoretical if it is Morita equivalent to a nilpotent fusion category. A fusion category is called solvable if it is Morita equivalent to a cyclically nilpotent fusion category.

\subsection{Braided fusion categories}
A braided fusion category $\C$ is a fusion category equipped with a natural isomorphism $c_{X,Y}:X\otimes Y\to Y\otimes X$ satisfying the hexagon axioms. Let $\C$ be a braided fusion category, and $\D\subseteq \C$ be a fusion subcategory. The M\"{u}ger centralizer $\D'$ of $\D$ in $\C$ is the category of all objects $Y\in \C$ such that $c_{Y,X}c_{X,Y}=id_{X\otimes Y}$ for all $X\in \D$. The centralizer $\D'$ is again a fusion subcategory of $\C$. The M\"{u}ger center of $\C$ is the M\"{u}ger centralizer $\mathcal{Z}_2(\C):=\C'$ of $\C$ itself. A braided fusion category $\C$ is called nondegenerate if its M\"{u}ger center $\mathcal{Z}_2(\C)$ is trivial.

\medbreak
A braided fusion category is called premodular if it has a spherical structure. Any weakly integral fusion category is premodular because it is automatically pseudounitary and has a canonical spherical structure with respect to which the categorical dimensions coincide with the Frobenius-Perron dimensions \cite[Proposition 8.23, 8.24]{etingof2005fusion}. A modular category is a nondegenerate premodular category. Hence, a weakly integral fusion category is modular if and only if it is nondegenerate.

\medbreak
If $\C$ is a modular category then the universal grading group $\mathcal{U}(\C)$ is isomorphic to the group $\Pic (\C)$ \cite[Theorem 6.2]{gelaki2008nilpotent}. In particular, $\FPdim (\C_{pt})=|\mathcal{U}(\C)|$.

\begin{prop}\label{prop22-1}
Let $\C$ be a weakly integral braided fusion category such that $\C_{pt}$ is not trivial. Then $\C$ has  one of the following properties:

(1)\quad The universal grading group $\mathcal{U}(\C)$ is not trivial;

(2)\quad The M\"{u}ger center $\mathcal{Z}_2(\C)$ is not trivial.
\end{prop}

\begin{proof}
We assume that part (1) does not hold, and show part (2). Assume that $\mathcal{Z}_2(\C)$ is trivial. Then $\C$ is a modular category. In this case, $\mathcal{U}(\C)$ is isomorphic to the group $\Pic (\C)$, which contradicts the assumptions. Thus $\mathcal{Z}_2(\C)$ is not trivial.
\end{proof}

Obviously, Proposition \ref{prop22-1} can be applied to weakly integral braided fusion categories which satisfy the assumption of Lemma \ref{lem21} or Proposition \ref{pro22}.

\medbreak
The following theorem is taken from \cite[Theorem 3.13, 3.14]{drinfeld2010braided}. In the case where $\C$ is modular, it is due to M\"{u}ger \cite{muger2003subfactors,muger2003structure}.

\begin{thm}\label{thm23}
Let $\C$ be a braided fusion category, and $\D\subseteq \C$ be a fusion subcategory.

(1)\quad $\FPdim (\D)\FPdim(\D')=\FPdim(\C)\FPdim(\D\cap\mathcal{Z}_2(\C))$.

(2)\quad If $\D$ is nondegenerate then there is an equivalence of braided fusion categories $\C\cong \D\boxtimes\D'$. In particular, $\D'$ is nondegenerate if and only if $\C$ is nondegenerate.
\end{thm}

A braided fusion category $\C$ is called slightly degenerate if its M\"{u}ger center $\mathcal{Z}_2(\C)$ is equivalent to the category $\SuperV$ of super vector spaces. The following theorem is taken from \cite[Proposition 2.6 and Corollary 2.8]{etingof2011weakly} and \cite[Lemma 5.4]{muger2000galois}.

\begin{thm}\label{thm24}
Let $\C$ be a braided fusion category.

(1)\quad If $\C$ is  slightly degenerate and pointed then $\C =\SuperV\boxtimes \C_0$, where $\C_0$ is a nondegenerate pointed category.

(2)\quad If $\C$ is a slightly degenerate integral category with $\FPdim(\C)>2$ then $\C$ contains an odd-dimensional simple object outside of the M\"{u}ger center $\mathcal{Z}_2(\C)$.

(3)\quad Supose that the M\"{u}ger center $\mathcal{Z}_2(\C)$
contains $\SuperV$. Let $g$ be the invertible
object generating $\SuperV$, and let $X$ be any simple object of $\C$. Then $g\otimes X$ is not isomorphic to $X$.
\end{thm}

Recall that a Tannakian fusion category is a symmetric fusion category such that it is equivalent to the representation category of a finite group as a symmetric fusion category.

\begin{lem}\cite[Corollary 2.50]{drinfeld2010braided}\label{lem25}
Let $\C$ be a symmetric fusion category. Then either $\C$ is Tannakian or $\C$ contains a Tannakian subcategory of Frobenius-Perron dimension $\frac{1}{2}\FPdim (\C)$. In particular, if $\FPdim (\C)$ is odd then $\C$ is Tannakian.
\end{lem}

\subsection{Equivariantizations}\label{subsec24}
Let $G$ be a finite group and $\mathcal{\C}$ be a fusion category. Let $\underline{G}$ denote the monoidal category whose objects are elements of $G$, morphisms are identities and tensor product is given by the multiplication in $G$. Let ${\rm\underline{Aut}}_{\otimes}\mathcal{\C}$ denote the monoidal category whose objects are tensor autoequivalences of $\C$, morphisms are isomorphisms of tensor functors and tensor product is given by the composition of functors.

\medbreak
An action of $G$ on $\mathcal{\C}$ is a monoidal functor
$$T:\underline{G}\to {\rm\underline{Aut}}_{\otimes}\mathcal{\C},\quad g\mapsto T_g$$
with the isomorphism $f^X_{g,h}: T_g(X)\otimes T_h(X)\cong T_{gh}(X)$, for every $X$ in $\mathcal{\C}$.

\medbreak
Let $\mathcal{\C}$ be a fusion category with an action of $G$. Then the fusion category $\mathcal{\C}^G$, called the $G$-equivariantization of $\mathcal{\C}$, is defined as follows \cite{bruguieres2000categories,drinfeld2010braided,muger2004galois}:

(1)\quad An object in $\mathcal{\C}^G$ is a pair $(X, (u^X_g)_{g\in G})$, where $X$ is an object of $\mathcal{\C}$ and $u^X_g: T_g(X)\to X$
is an isomorphism such that,
$$u^X_gT_g(u^X_h)= u^X_{gh}f^X_{g,h},\quad \mbox{for all\quad} g,h\in G.$$

(2)\quad A morphism $\phi: (Y,u_g^Y)\to (X,u_g^X)$ in $\mathcal{\C}^G$ is a morphism $\phi: Y\to X$ in $\mathcal{\C}$ such that $\phi u_g^Y=u_g^X\phi$, for all $g\in G$.

(3)\quad The tensor product in $\mathcal{\C}^G$ is  defined as $(Y,u_g^Y)\otimes (X,u_g^X)=(Y\otimes X, (u_g^Y\otimes u_g^X)j_g|_{Y,X})$, where $j_g|_{Y,X}:T_g(Y\otimes X)\to T_g(Y)\otimes T_g(X)$ is the isomorphism giving the monoidal structure on $T_g$.

\medbreak
There is a procedure opposite to equivariantization. Let $\C$ be  a fusion category and let $\Rep(G)\subseteq \mathcal{Z}(\C)$ be a Tannakian subcategory which embeds into $\C$ via the forgetful functor $\mathcal{Z}(\C)\to \C$, where $\mathcal{Z}(\C)$ is the Drinfeld center of $\C$. Let $A=Fun(G)$ be the algebra of function on $G$. It is a commutative algebra in $\mathcal{Z}(\C)$ under the embedding  above. Let $\C_G$ be the category of left $A$-modules in $\C$. It is a fusion category which is called the de-equivariantization of $\C$ by $\Rep(G)$.

The two procedures above are inverse to each other; that is, there are canonical equivalences $(\C_G)^G\cong \C$  and $(\C^G)_G\cong \C$. Moreover, we have $$\FPdim(\C^G)=|G|\FPdim(\C) \mbox{\,\,and\,\,} \FPdim(\C_G)=\frac{\FPdim(\C)}{|G|}.$$

\medbreak
Let $\C$ be a braided fusion category, and $\Rep(G)\subseteq \C$ be a Tannakian subcategory. The de-equivariantization $\C_G$ of $\C$ by $\Rep(G)$ is a braided $G$-crossed fusion category. The category $\C_G$ is not braided in general. But the neutral component $(\C_G)_e$ of the associated $G$-grading of $\C_G$ is braided. By \cite[Proposition 4.56]{drinfeld2010braided}, $\C$ is nondegenerate if and only if $(\C_G)_e$ is nondegenerate and the associated grading of $\C_G$ is faithful.

The following theorem will be used in Section \ref{sec3}.

\begin{thm}\cite[Theorem 7.2]{naidu2009fusion}\label{thm26}
Let $\C$ be a braided fusion category. Then $\C$ is group-theoretical if and only if it contains a Tannakian subcategory $\Rep(G)$ such that the de-equivariantization $\C_G$ by $\Rep(G)$ is pointed.
\end{thm}

\section{Fusion categories whose simple objects have a common prime divisor in their dimensions}\label{sec3}
\setcounter{equation}{0}
\begin{lem}\label{lem31}
Let $\C$ be a nonpointed integral braided fusion category such that the Frobenius-Perron dimensions of all noninvertible simple objects have a common prime divisor $p$. Suppose that $\FPdim(\C_{pt})=2$. Then $\C_{pt}$ is a Tannakian subcategory of $\C$.
\end{lem}

\begin{proof}
It is clear that $p=2$ in our case since $p$ divides $2$ by Lemma \ref{lem21}. It suffices to prove that $\C_{pt}$ is not equivalent to the category $\SuperV$ of super vector spaces.

\medbreak
Suppose on the contrary that $\C_{pt}$ is equivalent to $\SuperV$  and $g$ is the unique object generating $\C_{pt}$. Let $D$ be the M\"{u}ger centralizer of $\C_{pt}$ in $\C$. Then $\C_{pt}$ lies in the M\"{u}ger center of $D$. By Theorem \ref{thm23} (1), $6\leq\FPdim(\C)\leq\FPdim(\C_{pt})\FPdim(D)$, which implies that $\FPdim(D)>2$ and $D$ contains noninvertible objects of $\C$. Let $X$ be a noninvertible object of $D$. Then the decomposition of $X\otimes X^*$ (ref. Equation \ref{eq1}) shows that $g\otimes X\cong X$. It contradicts Theorem \ref{thm24}(3). So $\C_{pt}\cong\Rep(\mathbb{Z}_2)$ is Tannakian.
\end{proof}

\begin{thm}\label{thm32}
Let $\C$ be an integral braided fusion category such that the Frobenius-Perron dimensions of all noninvertible simple objects have a common prime divisor $p$. Then $\C_{pt}$ contains a nontrivial Tannakian subcategory of $\C$.
\end{thm}
\begin{proof}
Let $\D=\C_{pt}$ and $\mathbb{Z}_2(\D)$ be the M\"{u}ger center of $\D$. By Lemma \ref{lem21}, $\D$ is not trivial.

\medbreak
If $\mathbb{Z}_2(\D)$ is trivial then $\D$ is nondegenerate. So there is an equivalence of braided fusion categories $\C\cong\D\boxtimes\D'$ by Theorem \ref{thm23}(2), where $\D'$ is the M\"{u}ger centralizer of $\D$ in $\C$. Clearly, $\D'$ contains no nontrivial invertible objects but contains all noninvertible objects. This is impossible by Lemma \ref{lem21}. Therefore, $\D$ is degenerate.

\medbreak
If $\FPdim(\mathbb{Z}_2(\D))=2$ then either $\mathbb{Z}_2(\D)=\Rep(\mathbb{Z}_2)$ is the category of representations of $\mathbb{Z}_2$ as a symmetric category, or $\mathbb{Z}_2(\D)=\SuperV$ is the category of super vector spaces. If the first case holds true then we are done. If the second case holds true then $\D$ is slightly degenerate. In this case, $\D=\SuperV\boxtimes\D_0$ by Theorem \ref{thm24} (1), where $\D_0$ is a nondegenerate pointed category. If $\D_0$ is trivial then $\D=\SuperV$. By Lemma \ref{lem31}, such fusion category $\C$ can not exist. If $\D_0$ is not trivial then $\C\cong\D_0\boxtimes \D_0'$ by Theorem \ref{thm23}(2), where $\D_0'$ is the M\"{u}ger centralizer of $\D_0$ in $\C$. Comparing $\C\cong\D_0\boxtimes \D_0'$ and $\D=\SuperV\boxtimes\D_0$, we get that the subcategory of invertible objects of $\D_0'$ is $\SuperV$. Also, $\D_0'$ contains all noninvertible simple objects of $\C$. Such fusion category $\D_0'$ can not exist, also by Lemma \ref{lem31}.

\medbreak
Finally, we consider the possibility that $\FPdim(\mathbb{Z}_2(\D))>2$. In this case, $\mathbb{Z}_2(\D)$ must contain a nontrivial Tannakian subcategory by Lemma \ref{lem25}, so does $\C$.
\end{proof}

The following corollary follows from Theorem \ref{thm32} and the discussion in Section \ref{subsec24}.
\begin{cor}\label{cor32-1}
Let $\C$ be an integral braided fusion category such that the Frobenius-Perron dimensions of all noninvertible simple objects have a common prime divisor $p$. Then $\C$ is a $G$-equivariantization of a fusion category $\C_G$, where $G$ is an Abelian group and $\FPdim(\C_G)=\FPdim(\C)/|G|$.
\end{cor}

Recall that a near-group fusion category is a fusion category $\C$ where all but $1$ simple objects are invertible. The fusion rules of $\C$ are determined by a pair $(G,k)$, where $G$ is the group of invertible objects and $k$ is a nonnegative integer. Let $X$ be the unique noninvertible simple object. Then
\begin{equation}\label{eq6}
\begin{split}
X\otimes X=\bigoplus_{g\in G}g\oplus kX.
\end{split}\nonumber
\end{equation}

The following corollary follows from Corollary \ref{cor32-1}.
\begin{cor}\label{cor33}
Let $\C$ be an integral braided fusion category of type $(1,m;s,n)$. Then $\C$ is a $G$-equivariantization of a fusion category $\C_G$, where $G$ is an Abelian group. In particular, every integral braided near-group fusion category has this property.
\end{cor}

Using Theorem \ref{thm32}, we can recover a theorem in \cite{natale2013weakly}.

\begin{cor}\label{cor34}
Let $p$ be a prime number and let $\C$ be an integral braided fusion category such that the Frobenius-Perron dimensions of simple objects are powers of $p$. Then $\C$ is solvable.
\end{cor}
\begin{proof}
We prove by induction on $\FPdim(\C)$. By Theorem \ref{thm32}, $\C_{pt}$ contains a nontrivial Tannakian subcategory $\mathcal{E}=\Rep(G)$ of $\C$. Since $\mathcal{E}$ is pointed, $G$ is Abelian, and hence $G$ is solvable.

Recall from Section \ref{subsec24} that the de-equivariantization $\C_G$ of $\C$ by $\mathcal{E}$ is a $G$-grading fusion category. By \cite[Corollary 2.13]{burciu2013fusion}, the Frobenius-Perron dimensions of simple objects of $\C_G$ are powers of $p$. Let $(\C_G)_e$ be the neutral component of $\C_G$ associated to the $G$-grading. Then $(\C_G)_e$ is braided and the Frobenius-Perron dimensions of simple objects of $(\C_G)_e$ are powers of $p$. By induction, $(\C_G)_e$ is solvable. In addition, $\C_G$ is a $K$-extension of $(\C_G)_e$, where $K$ is a subgroup of $G$. Therefore, $\C$ is solvable since solvability is preserved under taking extensions and equivariantizations by  solvable groups \cite[Proposition 4.5]{etingof2011weakly}.
\end{proof}

\begin{rem}\label{rem32-2}
Let $\C$ be a weakly integral fusion category. Then $\FPdim(X)^2$ is an integer, for all $X\in \Irr(\C)$. If $\C$ is not integral then $\C$ is a $\mathbb{Z}_2$-extension of an integral fusion category.

Suppose that the Frobenius-Perron dimensions of all noninvertible simple objects of $\C$ have a common prime divisor $p$; that is, for every $X\in \Irr(\C)$ the dimension $\FPdim(X)$ has the form $pm^{\frac{n}{2}}$, where $m$ is a natural number and $n$ is a nonnegative integer. Then $\C$ is a $\mathbb{Z}_2$-extension of an integral fusion category $\D$. By Theorem \ref{thm32}, $\D$ contains a nontrivial Tannakian subcategory, so does $\C$. Similarly, we can get that all preceding results in this section also hold for weakly integral fusion categories.
\end{rem}

Group-theoretical categories can be explicitly described by finite group data  \cite{Ostrik2003}. So we are especially interested in when a fusion category is group-theoretical. We first recall some results from \cite{burciu2013fusion}.

\medbreak
Let $\C^G$ be the equivariantization of $\C$ under the action $T:\underline{G}\to {\rm\underline{Aut}}_{\otimes}\C$, and let $F:\C^G\to \C$ denote the forgetful functor. For every simple object $X$ of $\C$, let $G_X = \{g\in G : T_g(X)\cong X\}$ be the inertia subgroup of $X$. Since $X$ is a simple object, there is a $2$-cocycle
$\alpha_X: G_X \times G_X\to k^*$. The simple objects of $\C^G$ are parameterized by pairs $(X,\pi)$, where $X$ runs over the $G$-orbits on $\Irr(\C)$ and $\pi$ is an equivalence class of an irreducible $\alpha_X$-projective representation of $G_X$. If we use $S_{X,\pi}$ to indicate the isomorphism class of the simple object corresponding to the pair $(X,\pi)$, we get the following relation \cite[Corollary 2.13]{burciu2013fusion}:
\begin{equation}\label{eq7}
\begin{split}
\FPdim S_{X,\pi} ={\rm dim}\pi [G:G_X] \FPdim X.
\end{split}
\end{equation}

\begin{prop}\label{prop35}
Let $\C$ be an integral braided fusion category, and let $d$ be the least common multiple of Frobenius-Perron dimensions of all simple objects of $\C$. If there exists a nontrivial Tannakian subcategory $\Rep(G)\subseteq \C$ such that the de-equivariantization $\C_G$ is pointed then $d$ divides the order of $G$.
\end{prop}
\begin{proof}
Note that $\C\cong (\C_G)^G$ as a fusion category. Let $S_{X,\pi}$ be a simple object of $\C$ as described above. Since $\C_G$ is pointed, $\FPdim X=1$. Hence $\FPdim S_{X,\pi} ={\rm dim}\pi [G:G_X]$ by Equation \ref{eq7}. On the other hand, ${\rm dim}\pi [G:G_X]$ divides the order of $G$. So $\FPdim S_{X,\pi}$ divides the order of $G$, and hence $d$ divides the order of $G$.
\end{proof}

Note that under the assumption of Proposition \ref{prop35} the fusion category $\C$ is group-theoretical, by Theorem \ref{thm26}.  Motivated by this observation, we get the following theorem.

\begin{thm}\label{thm36}
Let $p$ be a prime number and let $\C$ be a braided fusion category such that the Frobenius-Perron dimensions of simple objects are powers of $p$. Suppose that $\FPdim(\C_{pt})=p$. Then $\C$ is group-theoretical. Moreover, $\C$ is of type $(1,p;p,m)$ or $(1,p)$.
\end{thm}

\begin{proof}
By Theorem \ref{thm32}, $\C_{pt}\cong \Rep(\mathbb{Z}_p)$ is a Tannakian subcategory of $\C$. Let $\C_{\mathbb{Z}_p}$ be the de-equivariantization of $\C$ by $\Rep(\mathbb{Z}_p)$. Then
\begin{equation}\label{eq8}
\begin{split}
\FPdim \C_{\mathbb{Z}_p} =\frac{\FPdim \C}{p}=1+pm,
\end{split}\nonumber
\end{equation}
for some positive integer $m$.

\medbreak
By Equation \ref{eq7}, the Frobenius-Perron dimension of every simple object of $\C_{\mathbb{Z}_p}$ is a power of $p$. If $\C_{\mathbb{Z}_p}$ contains a noninvertible simple object then Lemma \ref{lem21} shows that $p$ divides the Frobenius-Perron dimension of $(\C_{\mathbb{Z}_p})_{pt}$. Hence $\FPdim \C_{\mathbb{Z}_p} = np+p^2s$ for some positive integers $m$ and $s$. It is a contradiction. So every simple object of $\C_{\mathbb{Z}_p}$ is invertible, that is, $\C_{\mathbb{Z}_p}$ is pointed. It follows from Theorem \ref{thm26}(2) that $\C$, being an equivariantization of a pointed fusion category, is group-theoretical.

\medbreak
Suppose that $\C$ contains a simple object $S$ of Frobenius-Perron dimension $p^i, i\geq2$. Then $S$ corresponds to a pair $(X,\pi)$, where $X$ is a simple object of $\C_G$, $\pi$ is an equivalence class of an irreducible $\alpha_X$-projective representation of $G_X$ as recalled before Proposition \ref{prop35}, and we have
$$\FPdim S = [G:G_X] {\rm dim}\pi \FPdim X.$$

Since $[G:G_X] {\rm dim}\pi$ divides $p$, $\FPdim X=p^i$ or $p^{i-1}$. Therefore, $\C_G$ can not be pointed. It is a contradiction. So $\C$ can not contain simple objects of Frobenius-Perron dimension $p^i, i\geq2$.
\end{proof}

\begin{rem}
In \cite{natale2012fusion}, Natale and Plavnik studied semisimple Hopf algebras of type $(1,p;p,n)$, where $p$ is a prime number. They proved that if $H$ is a quasitriangular semisimple Hopf algebra of type $(1,p;p,n)$ as an algebra then $\Rep H$ is solvable.
\end{rem}


%
%

\section{Fusion categories of type $(1,m;p,n;q,s)$}\label{sec4}
\setcounter{equation}{0}
\begin{lem}\label{lem41}
Let $1<p<q$ be positive integers such that ${\rm gcd}(p,q)=1$, and let $\C$ be a braided fusion category of type $(1,2;p,m;q,n)$. Then $\C$ contains a nontrivial Tannakian subcategory in any of the following cases:

(1)\quad $p^2\leq q$;

(2)\quad $p,q$ are powers of two distinct prime numbers.
\end{lem}
\begin{proof}

Suppose that $\C_{pt}$ is equivalent to $\SuperV$, otherwise $\C_{pt}\cong\Rep(\mathbb{Z}_2)$ is a nontrivial Tannakian subcategory of $\C$. Let $D$ be the M\"{u}ger centralizer of $\C_{pt}$ in $\C$. Then $\C_{pt}$ lies in the M\"{u}ger center of $D$. We only consider the case when $D$ is slightly degenerate, otherwise the M\"{u}ger center $Z_2(\D)$ is bigger than $\SuperV$ and contains a nontrivial Tannakian subcategory by Lemma \ref{lem25}.

By Theorem \ref{thm23} (1), $\FPdim(\C_{pt})\FPdim(\D)\geq\FPdim(\C)> 6$, which implies that $\FPdim(\D)>2$ and hence $\D$ contains noninvertible objects of $\C$. Since $\D$ is slightly degenerate and $\D_{pt}$ contains only $2$ elements, Theorem \ref{thm24} (3) shows that the stabilizer $G[X]$ is trivial for all $X\in \Irr(\C)$. Thus $\D$ can not be of type $(1,2;p,m')$ and $(1,2;q,n')$. This implies that $\D$ must be of type $(1,2;p,m';q,n')$.

(1)\quad Let $X\in \Irr_p(\D)$. Then $X\otimes X^*$ can not contain elements of $\Irr_q(\D)$ since $p^2\leq q$. This is impossible because $G[X]$ is trivial. Thus $\D_{pt}=\C_{pt}$ is a nontrivial Tannakian subcategory of $\C$.

(2)\quad If $p,q$ are powers of two distinct prime numbers, then \cite[Proposition 7.4]{etingof2011weakly} shows that $\D$ contains a nontrivial Tannakian subcategory, and so does $\C$.
\end{proof}

\begin{thm}\label{thm42}
Let $1<p<q$ be positive integers such that ${\rm gcd}(p,q)=1$, and let $\C$ be of type $(1,m;p,n;q,s)$. Then $\C$ contains a nontrivial Tannakian subcategory in any of the following cases:

(i)\quad $p^2\leq q$;

(ii)\quad $p,q$ are powers of two distinct prime numbers.
\end{thm}
\begin{proof}
 By Proposition \ref{pro22}, $\C_{pt}$ is not trivial. Let $D=\C_{pt}$ and $\mathbb{Z}_2(D)$ be the M\"{u}ger center of $\D$.

 \medbreak
 If $\FPdim(\mathbb{Z}_2(D))>2$ then $\mathbb{Z}_2(D)$ contains a nontrivial Tannakian subcategory by Lemma \ref{lem25}.

 \medbreak
  If $\FPdim(\mathbb{Z}_2(D))=1$ then $D$ is nondegenerate. There is an equivalence of braided fusion categories $\C\cong \D\boxtimes\D'$, where $\D'$ is the M\"{u}ger centralizer of $\D$ in $\C$. Then $\D'$ is of type $(1,1;p,n;q,s)$. This is impossible by Proposition \ref{pro22}.

\medbreak
 Finally, we consider the case when $\FPdim(\mathbb{Z}_2(D))=2$ and $\mathbb{Z}_2(D)$ is equivalent to $\SuperV$. In this case, $\D$ is slightly degenerate. Then $D\cong\SuperV\boxtimes \D_0$ by Theorem \ref{thm24}(1), where $\D_0$ is a nondegenerate pointed fusion category. If $\D_0$ is trivial then $\D\cong\SuperV$ and $\C$ is of type $(1,2;p,m;q,s)$. So we are done in this case by Lemma \ref{lem41}. If $\D_0$ is not trivial then $\C\cong\D_0\otimes \D_0'$ as a braided fusion category by Theorem \ref{thm23}, where $\D_0'$ is the M\"{u}ger centralizer of $\D_0$ in $\C$. In particular, $\D_0'$ is of type  $(1,2;p,m;q,s)$. So $\D_0'$ contains a nontrivial Tannakian subcategory by Lemma \ref{lem41}.
\end{proof}

\begin{cor}\label{cor43}
Let $\C$ be a braided fusion category of type $(1,m;p,n;q,s)$, where $p<q$ are prime numbers. Then $\C$ is solvable.
\end{cor}
\begin{proof}
By Theorem \ref{thm42}, $\C$ contains a nontrivial Tannakian subcategory $\Rep(G)$. Since the Frobenius-Perron dimensions of simple objects of $\Rep(G)$ are $1,p$ or $q$, we have that $G$ is solvable. We then prove the corollary by induction on $\FPdim (\C)$.

\medbreak
Let $\C_G$ be the de-equivariantization of $\C$ by $\Rep(G)$, and $(\C_G)_e$ be the neutral component of the associated $G$-grading of $\C_G$. The  Frobenius-Perron dimensions of simple objects of $\C_G$ are $1,p$ or $q$ by Equation \ref{eq7}. If $(\C_G)_e$ is of type $(1,m')$, $(1,m';p,n')$ or $(1,m';q,n')$ then $(\C_G)_e$ is solvable by Corollary \ref{cor34}. If $(\C_G)_e$ is of type $(1,m';p,n';q,s')$ then $(\C_G)_e$ is solvable by induction.  In addition, $\C_G$ is a $K$-extension of $(\C_G)_e$, where $K$ is a subgroup of $G$. Therefore, $\C$ is solvable since solvability is preserved under taking extensions and equivariantizations by  solvable groups \cite[Proposition 4.5]{etingof2011weakly}.
\end{proof}

\begin{rem}
Let $\C$ be a braided fusion category of type $(1,m;p,n;q,s)$. It is hard to determine when $\C$ is group-theoretical, even though in the simplest case when $p=2$. In fact, there exists a braided category of such type which is not group-theoretical. As we have mentioned in the Introduction, Gelaki, Naidu and Nikshych \cite[Section 5C]{gelaki2009centers} constructed a class of modular categories which is of type $(1,2;2,\frac{1}{2}(q^2-1);q,2)$. They are not group-theoretical.
\end{rem}

\section{Applications}
We will apply the results obtained so far to integral braided fusion categories, and prove the following result.

\begin{thm}\label{thm51}
Let $\C$ be an integral braided fusion category. Suppose that $\FPdim(\C)$ is odd and $\FPdim(\C)<765$. Then $\C$ is solvable.
\end{thm}

Before giving the proof of Theorem \ref{thm51}, we need some preparations. The following lemma is a special case of \cite[Corollary 8.2]{Ng200734}.

\begin{lem}\label{lem52}
Let $\C$ be a weakly integral fusion category. Suppose that there exists a nontrivial simple object $X$ of $\C$ such that $X^*\cong X$. Then $\FPdim(\C)$ is even.
\end{lem}

Based on the above lemma, we can get the following result. The argument of its proof is
word for word as in the proof of \cite[Theorem 5.1]{kashina2002self}, so it is omitted.

\begin{lem}\label{lem53}
Let $\C$ be an integral fusion category. Suppose that there exists a nontrivial simple object $X$ of $\C$ such that $\FPdim(X)$ is even. Then $\FPdim(\C)$ is even.
\end{lem}

The following lemma is essential in the proof of Theorem \ref{thm51}.

\begin{lem}\label{lem54}
Let $\C$ be an integral braided fusion category.

(1)\quad If $d_i$ $(2\leq i\leq s)$ does not divide $\FPdim(\C)$ then $\C$ cannot be of type
$(1,n_1;d_2,n_2;\cdot\cdot\cdot;d_s,n_s)$.

(2)\quad If  $\FPdim(\C)$ is odd then $\C$ cannot be of type
$(1,n_1;d_2,n_2;\cdots;d_s,n_s)$, where there exists
$i\in \{2,\cdots,s\}$ such that $n_i$ is odd or $d_i$ is even.

(3)\quad If $n_1$ does not divide $\FPdim(\C)$ or $ n_id_i^2$ then $\C$ cannot be of type
$(1,n_1;d_2,n_2;\cdots;d_s,n_s)$.

(4)\quad If  $\FPdim(\C)$ is odd then $\C$ cannot be of type
$(1,m;3,n;\cdots)$, where $m$ is not divisible by $3$.

(5)\quad If  $\FPdim(\C)$ is odd and $3+9n$ does not divide $\FPdim(\C)$ then $\C$ cannot be of type $(1,3;3,n;\cdots)$.

(6)\quad If  $\FPdim(\C)$ is odd, $\C$ does not have simple objects
of dimension $9$ and $m+9n$ does not divide $\FPdim(\C)$ then $\C$
cannot be of type $(1,m;3,n;\cdots)$.

(7)\quad If  $\FPdim(\C)$ is odd, $\C$ does not have simple objects
of dimension $3,7$ and $5$ does not divide $m$ then $\C$
cannot be of type $(1,m;5,n;\cdots)$.

(8)\quad If $t$ does not divide $m$ then $\C$ can not be of type
$(1,m;t,n)$.


(9)\quad If $\C$ is of type $(1,n_1;d_2,n_2;d_3,n_3)$ then $n_1>1$.
\end{lem}

\begin{proof}
Part (1) follows from \cite[Theorem 2.11]{etingof2011weakly}. Parts (2)--(8) follow from \cite[Lemma 5.2]{dong2012further}. In fact, the proof of \cite[Lemma 5.2]{dong2012further} only uses the properties of the Grothendieck ring of a semisimple Hopf algebra and Lemma \ref{lem52}, \ref{lem53} (in the semisimple Hopf algebra setting). Therefore, the proof of \cite[Lemma 5.2]{dong2012further}  also works  \textit{mutatis mutandis} in the integral fusion category setting. Part (9) follows from Lemma \ref{lem21} and Proposition \ref{pro22}.
\end{proof}

\begin{proof}[Proof of Theorem \ref{thm51}]
Let $p,q$ and $r$ be distinct prime numbers. Integral fusion categories of Frobenius-Perron dimension $pqr$ are classified in \cite{etingof2011weakly}. These fusion categories are group-theoretical. In the same paper, integral fusion categories of Frobenius-Perron  dimension $p^mq^n$ are proved solvable, where $m,n$ are non-negative integers. Therefore, it suffices to consider the case that $\FPdim(\C)= 315,495,525,585,693$ and $735$.

It is easy to write a computer program by which one finds out all possible positive integers $1=d_1,d_2,\cdots,d_s$ and $n_1,n_2,\cdots,n_s$ such that
$\FPdim(\C)=\sum_{i=1}^{s}n_id_i^2$, and then one can eliminate those which
can not be types of $\C$ by using Lemma \ref{lem54}. Hence, we can get the following data.

If $\FPdim(\C)=315$ then $\C$ is of one of the following types:
\begin{align*} &(1,63;3,28), (1,45;3,30), (1,9;3,34), (1,15;5,12), (1,21;7,6),
\\ &(1,3;3,2;7,6), (1,9;3,16;9,2).
\end{align*}

If $\FPdim(\C)=495$ then $\C$ is of one of the following types:
\begin{align*}&(1,9;3,54), (1,45;3,50), (1,99;3,44), (1,45;5,18), (1,11;11,4),\\  &(1,9;9,6), (1,9;3,18;9,4), (1,9;3,36;9,2),\\
&(1,9;3,4;5,18), (1,15;3,20;5,12),\\
&(1,45;15,2), (1,9;3,4;15,2).
\end{align*}

If $\FPdim(\C)=525$ then $\C$ is of one of the following types:
\begin{align*}&(1,21;3,56),(1,75;3,50), (1,3;3,58), (1,175;5,14), (1,75;5,18),\\ &(1,25;5,20),(1,35;7,10),\\ &(1,3;3,8;5,18),\\
&(1,75;15,2), (1,25;5,2;15,2), (1,3;3,8;15,2).
\end{align*}

If $\FPdim(\C)=585$ then $\C$ is of one of the following types:
\begin{align*}&(1,117;3,52), (1,9;3,64), (1,45;3,60),
(1,9;3,10;9,6),(1,9;3,28;9,4),\\ & (1,9;3,46;9,2).
 \end{align*}

If $\FPdim(\C)=693$ then $\C$ is of one of the following types:
\begin{align*}&(1,9;3,76), (1,63;3,70), (1,99;3,66), (1,7;7,14)(1,9;3,4;9,8),\\ &(1,9;3,22;9,6), (1,9;3,40;9,4), (1,9;3,58;9,2),\\
&(1,21;3,42;7,6).
\end{align*}

If $\FPdim(\C)=735$ then $\C$ is of one of the following types:
\begin{align*} &(1,15;3,80), (1,105;3,70), (1,35;5,28), (1,147;7,12),(1,49;7,14),\\&(1,245;7,10), \\
&(1,3;3,16;7,12), (1,21;3,14;7,12), (1,15;3,30;5,18),\\
&(1,15;3,30;15,2).
\end{align*}

 Now, there are three possibilities: if the Frobenius-Perron dimensions of simple objects of $\C$ are powers of a prime then $\C$ is solvable by Corollary \ref{cor34}; if the Frobenius-Perron dimensions of simple objects of $\C$ are two prime numbers then $\C$ is solvable by Corollary \ref{cor43}; if the Frobenius-Perron dimensions of simple objects of $\C$ has a common prime divisor then $\C$ has a nontrivial Tannakian subcategory $\E=\Rep(G)$, where $G$ is an Abelian group by Theorem \ref{thm32}. In the last case, $\C$ is a $G$-equivariantization of a fusion category $C_G$  and $C_G$ is a $K$-extension of a braided fusion category $(C_G)_e$, where $K$ is a subgroup of $G$. Counting dimension, $(C_G)_e$ is solvable. Therefore, $\C$ is solvable by \cite[Proposition 4.5]{etingof2011weakly}. This finishes the proof.
\end{proof}

\section{Acknowledgements}
The paper was written during the first author visited University of Southern California. He greatly appreciates Professor Susan Montgomery and the Department of Mathematics of University of Southern California for their warm hospitality and wonderful office environment. His  visit is financially supported by the China Scholarship Council and  Nanjing Agricultural University. The authors are grateful to Huali Huang, Gongxiang Liu and Sonia Natale for valuable discussions on fusion categories and finite groups. The authors' research is supported by the Natural Science Foundation of China (11201231).

\end{document}